\documentclass{amsart}[12pt]
\usepackage[dvipsone]{graphicx}
\usepackage{epsfig,amsmath,amssymb}
\usepackage[all]{xy}
\CompileMatrices
\newdir{ >}{{}*!/-5pt/\dir{>}}

\newcommand{\lga}{\longrightarrow}
\newcommand{\lgaf}{\longleftarrow}
\newcommand{\limin}{\displaystyle \lim_{\longleftarrow}}

\newcommand {\Z}{\mathbb Z}
\newcommand {\R}{\mathbb R}
\newcommand {\N}{\mathbb N}
\newcommand {\sub}{\subset}

\newtheorem{theorem}{Theorem}[section]
\newtheorem{lemma}[theorem]{Lemma}

\newtheorem{proposition}[theorem]{Proposition}
\newtheorem{corollary}[theorem]{Corollary}
\newtheorem{conjecture}[theorem]{Conjecture}

\theoremstyle{definition}
\newtheorem{definition}[theorem]{Definition}

\theoremstyle{remark}
\newtheorem{remark}[theorem]{Remark}

\numberwithin{equation}{section}

\begin{document}

\title{One-relator groups and proper $3$-realizability}


\author{M. C\'ardenas, F.F. Lasheras, A. Quintero}
\address{Departamento de Geometr\'{\i}a y Topolog\'{\i}a, Universidad de Sevilla,
Apdo 1160, 41080-Sevilla, Spain}
\email{mcard@us.es, lasheras@us.es, quintero@us.es}
\author{D. Repov{\v{s}}}
\address{Institute of Mathematics, Physics and Mechanics, University of Ljubljana,
P.O. Box 2964, Ljubljana 1001, Slovenia}
\email{dusan.repovs@fmf.uni-lj.si}



\date{}



\begin{abstract}
How different is the universal cover of a given finite $2$-complex
from a $3$-manifold (from the proper homotopy viewpoint)?
Regarding this question, we recall that a finitely presented group
$G$ is said to be properly $3$-realizable if there exists a
compact $2$-polyhedron $K$ with $\pi_1(K) \cong G$ whose universal
cover $\tilde{K}$ has the proper homotopy type of a PL
$3$-manifold (with boun\-dary). In this paper, we study the
asymptotic behavior of finitely generated one-relator groups and
show that those having finitely many ends are properly
$3$-realizable, by describing what the fundamental pro-group looks
like, showing a property of one-relator groups which is stronger
than the QSF property of Brick (from the proper homotopy
viewpoint) and giving an alternative proof of the fact that
one-relator groups are semistable at infinity.
\end{abstract}

\maketitle

\section{Introduction}
The following question was formulated in \cite{L2} for an
arbitrary finitely presented group $G$ : {\it does there exist a
compact $2$-polyhedron $K$ with $\pi_1(K) \cong G$ whose universal
cover $\tilde{K}$ is proper homotopy equivalent to a $3$-manifold
?} If so, the group $G$ is said to be {\it properly
$3$-realizable}. Recall that two spaces are said to be proper
homotopy equivalent if they are homotopy equivalent and all
homotopies involved are proper maps, i.e., they have the property
that the inverse image of every compact subset is compact. It is a
fact that one can get the above universal cover $\tilde{K}$ of $K$
proper homotopy equivalent to a $4$-manifold, as one can take $K$
to be the Cayley complex associated to a (finite) group
presentation of $G$, which can easily be embeded in ${\R}^4$.
Moreover, it is known that the proper homotopy type of any locally
finite $2$-dimensional $CW$-complex can be represented by a
subpolyhedron in ${\R}^4$ (see \cite{CFLQ}). The question whether
or not one can do better than this, i.e., whether or not a given
finitely presented group $G$ is properly $3$-realizable, is of
interest as it has implications in the theory of cohomology of
groups : if $G$ is properly $3$-realizable then for some
(equivalently any) compact $2$-polyhedron $K$ with $\pi_1(K) \cong
G$ the group $H^2_c(\tilde{K};{\Z})$ is free abelian (by manifold
duality arguments), and hence so is $H^2(G;{\Z}G)$ (see
\cite{GM}). It is a long standing conjecture (attributed to Hopf)
that $H^2(G;{\Z}G)$ is free abelian for every finitely presented
group $G$. Observe that there are examples of locally compact,
simply connected $2$-polyhedra $X$ for which $H^2_c(X;{\Z})$ is
not free abelian, but they are known not to be covering spaces of
compact polyhedra (see \cite{CFLQ, L2}). There are several results
in the literature regarding properly $3$-realizable groups (see
\cite{ACLQ,CL2,CLQR,CLR,L4}). An example of a non-$3$-manifold
group which has this property is already pointed out in \cite{L2}.
 It is worth mentioning that we have recently found that there are
groups which are $1$-ended and semistable at infinity but not
properly $3$-realizable (see \cite{C et al, FLR}). The main result of
the present paper is:
\begin{theorem} Every finitely generated one-relator group $G$ with finitely many ends is properly $3$-realizable.
\end{theorem}
\indent Recall that, given a compact $2$-polyhedron $K$ with
$\pi_1(K) \cong G$ and $\tilde{K}$ as its universal cover, the
number of ends of $G$ equals the number of ends of $\tilde{K}$
which in turn equals $0,1, 2$ or $\infty$ (see \cite{Geo,SWa}).

\begin{proof}[Proof of Theorem 1.1] Given a one-relator group $G$ and a
presentation $P$ of $G$ with a single defining relation, it is
shown in Proposition \ref{prop} below that we can alter (within
its proper homotopy type) the universal cover $\tilde{K}_P$ of the
standard $2$-complex $K_P$ associated to this presentation so as
to get a new $2$-complex $\widehat{K}_P$ together with a
filtration $\widehat{C}_1 \sub \widehat{C}_2 \sub \cdots \sub
\widehat{K}_P$ of compact simply connected subcomplexes such that
(for any given base ray in $\widehat{K}_P$) the tower of groups,
$pro-\pi_1(\widehat{K}_P)$,
\[
\{1\} \leftarrow \pi_1(\widehat{K}_P - int(\widehat{C}_1))
\leftarrow \pi_1(\widehat{K}_P - int(\widehat{C}_2)) \leftarrow
\cdots
\]
is a telescopic tower, i.e., it is a tower of finitely generated
free groups of increasing bases where the bonding maps are
projections (see $\S \ref{appendix})$. Thus,
$pro-\pi_1(\tilde{K}_P)$ (and hence the fundamental pro-group of
$G$) is also of that type, up to pro-isomorphism. In the $1$-ended
case it is known that a group with such a fundamental pro-group is
properly $3$-realizable (see (\cite{L4}, Thm. 1.2)). Note that we
are already done in the $0$-ended or $2$-ended case as the
$0$-ended case corresponds to finite cyclic groups, and in the
$2$-ended case we only have to deal with the group ${\Z}$ of
integers, which is the only one-relator group having ${\Z}$ as a
subgroup of finite index. See also (\cite{ACLQ}, Cor. 1.2).
\end{proof}
\begin{remark} Observe that the proof of Theorem 1.1 given above shows
that the group $G$ is semistable at infinity (cf. \cite{MT}). See
$\S \ref{appendix}$.
\end{remark}
\indent It is worth noting that if $P$ is any finite group
presentation of $G$ with a single defining relation and $K_P$ is
the standard $2$-complex (with a single vertex and a single
$2$-cell) associated to this presentation, then Theo\-rem 1.1
together with (\cite{ACLQ}, Prop. 1.3) yields that the universal
cover of $K_P \vee S^2$ is proper homotopy equivalent to a
$3$-manifold. In fact, we conjecture that we may disregard the
$2$-sphere, having the universal cover of $K_P$ proper homotopy
equivalent to a $3$-manifold itself.
\begin{corollary} Every torsion-free finitely generated one-relator group $G$ is properly $3$-realizable.
\end{corollary}
\begin{proof} Observe that, in general, if a group has infinitely many ends
then Stallings' theo\-rem tells us that it splits as an
amalgamated free product or an HNN-extension over a finite group
(see \cite{SWa, Geo}), and Dunwoody's accessibility result
\cite{D} shows that the process of further factorization of the
group in this way must terminate after a finite number of steps,
and each of the factors can have at most one end. In the
torsion-free case, the above translates into a decomposition of
$G$ into a free product of a free group with a one-relator group,
the latter having at most one end. The conclusion now follows from
Theorem 1.1 and the fact that free groups are properly
$3$-realizable, and free pro\-ducts of properly $3$-realizable
groups are again properly $3$-realizable (see (\cite{ACLQ}, Thm.
1.4)).
\end{proof}
\indent On the other hand, a conjecture of the following type was
stated in \cite{FP}: {\it if a torsion-free one relator group
decomposes as a certain amalgamated free product $A*_CB$ over a
free group $C$, then each of the factors $A$ and $B$ must be
either a free group or a one-relator group}. In fact, it was
proved in \cite{FP} that this is so from the homology viewpoint.
We hereby pose the following conjecture :
\begin{conjecture} If a one-relator group with torsion
decomposes as an amalgamated free product $A*_CB$ (resp. an
HNN-extension $A*_C$) over a finite group $C$ (which is
necessarily cyclic, see \cite{KS,LS}), then each of the factors
$A$ and $B$ (resp. the base group $A$) must be either a
one-relator group or a free product of one-relator groups (with
torsion).
\end{conjecture}
\indent Observe that if this conjecture is true then, using the
results of Stallings and Dunwoody (as in Corollary 1.2), the
problem of showing that all finitely generated one-relator groups
are properly $3$-realizable can be reduced to the $1$-ended case,
by (\cite{ACLQ}, Thm. 1.4), and hence completely solved by Theorem
1.1. Thus, we conjecture the following :
\begin{conjecture} The ``finitely many ended" hypothesis in
Theorem 1.1 can be omitted.
\end{conjecture}
\section{One-relator groups and their structure at infinity}
The purpose of this section is to obtain some asymptotic
properties of one-relator groups, which will be essential for our
proof of Theorem 1.1. From now on, all complexes will be assumed
to be PL
CW-complexes (in the sense of (\cite{HMS}, $\S 1.4$)).\\
\indent Given a finitely presented group $G$ and a finite
$2$-dimensional CW-complex $X$ with $\pi_1(X) \cong G$, we recall
that $G$ is said to be QSF (i.e., {\it quasi simply filtered}) if
the universal cover $\tilde{X}$ of $X$ admits an exhaustion which
can be ``approximated" by finite simply connected CW-complexes,
i.e., for every finite subcomplex $A \sub \tilde{X}$ there is a
cellular map $f : Y \lga \tilde{X}$ from a finite simply connected
CW-complex $Y$ which is a homeomorphism on $f^{-1}(A)$. On the
other hand, $G$ is said to be WGSC (i.e., {\it weakly
geometrically simply connected}) if $\tilde{X}$ has in
fact an exhaustion by finite simply connected subcomplexes.\\
\indent It was shown in \cite{BM} that all one-relator groups are
QSF. Furthermore, the properties QSF and WGSC have recently been
shown to be equivalent (see \cite{FG,FO}). Next, in Proposition
\ref{prop} below, we shall demonstrate a stronger property of
one-relator groups (from the proper homotopy viewpoint) giving
also an alternative proof of the fact that one-relator groups are
semistable at infinity (see \cite{MT}). We shall need the
following :
\begin{definition} Let $X$ be a $2$-dimensional (PL) CW-complex and let
$d^n, e^{n+1} \sub X$ be cells of $X$, $d$ being a free face of
$e$ when considered as a subcomplex of $X$. An {\it elementary
internal collapse $(e,d)$} in $X$ consists of ``collapsing" the
cell $e$ through its face $d$, even if $d$ is not a free face of
$e$ within the entire complex $X$. Of course, this implies
dragging all the material adjacent to that face thus producing a
new CW-complex proper homotopy equivalent to $X$. We will refer to
any of the inverses of an elementary internal collapse as an {\it
elementary internal expansion}.
\end{definition}
\begin{definition} We say that two (possibly non-compact) $2$-complexes $X$ and $Y$
are {\it strongly proper homotopy equivalent} if one is obtained
from the other by a (possibly infinite) sequence of elementary
internal collapses and/or expansions, in such a way that the
resulting homotopy equivalence is a proper homotopy equivalence.
\end{definition}
\begin{definition} Let $X$ be a $2$-dimensional simply connected
CW-complex. We will say that a filtration $C_1 \sub C_2 \sub
\cdots \sub X$ of compact subcomplexes is {\it nice} if each $C_n$
is simply connected and, for any choice of base ray $[0, \infty)
\lga X$, the fundamental pro-group
\[
\{1\} \leftarrow \pi_1(X - int(C_1)) \leftarrow \pi_1(X -
int(C_2)) \leftarrow \cdots
\]
is a tower of finitely generated free groups of increasing bases
where the bonding maps are projections.
\end{definition}
\begin{lemma} \label{alter} Let $X$ be a $2$-dimensional simply connected
CW-complex together with a nice filtration $C_1 \sub C_2 \sub
\cdots \sub X$, and let $T_i, i \in I$, be a (locally finite)
collection of trees inside $X$. We can get a new $2$-complex
$\widehat{X}$ (strongly proper homotopy equivalent to $X$) still
containing the $T_i$'s and a nice filtration $\widehat{C}_1 \sub
\widehat{C}_2 \sub \cdots \sub \widehat{X}$ such that each
intersection $\widehat{C}_n \cap T_i$ ($n \geq 1, i \in I$) is
either empty or a connected subtree (and hence contractible).
\end{lemma}
\begin{proof} We shall reroute certain $2$-cells of $C_i$ up on ``bridges"
doubling a subforest of $\bigcup T_j$ in order to get connected
intersections $C_i \cap T_j$. For this, consider $C_1 \sub X$ and
let $T_{i_1}, \dots, T_{i_r} \sub X$ be those trees of the given
collection which intersect $C_1$. We denote by $Z_{1,m} \sub
T_{i_m}$ ($1 \leq m \leq r$) the smallest connected subtree
containing $C_1 \cap T_{i_m}$, and let $n(1) \geq 1$ be such that
$Z_{1,m} \sub C_{n(1)}$, for every $1 \leq m \leq r$. Let
$T_{i_1}, \dots , T_{i_r}, T_{i_{r+1}}, \dots, T_{i_s}$ be those
trees of the collection which intersect $C_{n(1)}$, and take
$Z_{n(1),m} \sub T_{i_m}$ to be either the connected subtree
satisfying $Z_{1,m} \sub Z_{n(1),m} \sub C_{n(1)} \cap T_{i_m}$ if
$1 \leq m \leq r$, or any component of $C_{n(1)} \cap T_{i_m}$
otherwise. We will perform in $X$ a sequence of elementary
internal expansions, one for each $k$-cell ($k=0,1$) of $
\displaystyle \Gamma = \bigcup_{m=1}^s ( C_{n(1)} \cap T_{i_m} -
Z_{n(1),m} )$ as follows. First, we introduce for each $1$-cell
$d^1 \sub \Gamma$ a new $2$-cell $e^2$ having its boundary divided
into two arcs, one of them corresponding to the old $d^1$ and the
other being such that every $2$-cell of $C_{n(1)}$ containing
$d^1$ in the original complex is now being attached along this new
arc. Any other $2$-cell containing $d^1$ in the original complex
is still
attached along $d^1$ (see figure 1).\\
\begin{figure}
\centerline{\psfig{figure=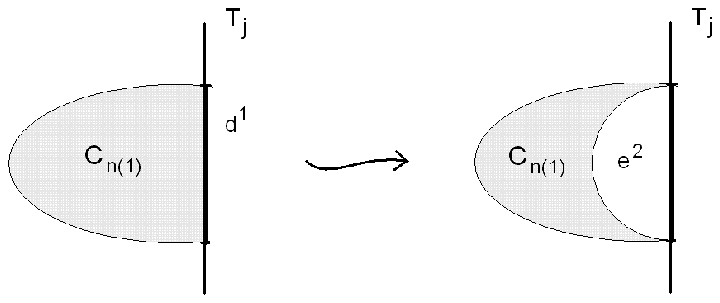,height=4.2cm,width=8cm}}
\vspace{-10mm} \label{figure1} \caption{}
\end{figure}
\begin{figure}
\centerline{\psfig{figure=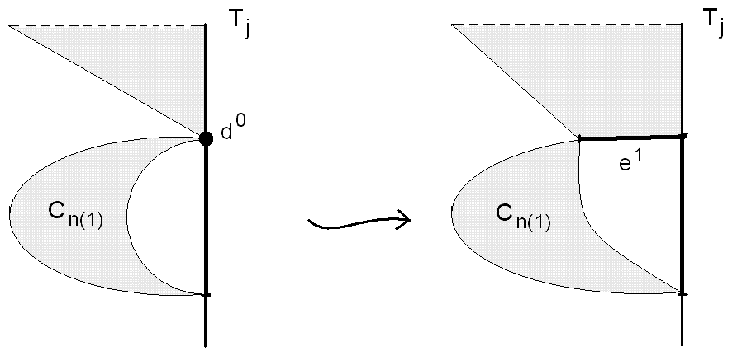,height=4.2cm,width=8cm}}
\vspace{-10mm} \label{figure2} \caption{}
\end{figure}
\indent These elementary internal expansions (leaving the trees
$T_{i_m}$ unaltered) yield a complex $X'$ in which $C_{n(1)}$ is
transformed into a new subcomplex $C'_{n(1)} \sub X'$ which meets
each $T_{i_m} \sub X'$ outside $Z_{n(1),m}$ in a finite number of
vertices in $\Gamma \cap T_{i_m}$. Next, we introduce for each of
these vertices $d^0 \in \Gamma \cap T_{i_m}$ a new $1$-cell $e^1$
with distinct boundary points, one of them corresponding to the
old $d^0$, and such that every $2$-cell outside $C'_{n(1)}$
containing a $1$-cell of $T_{i_m}$ adjacent to $d^0$ is now also
being attached along $e^1$, as shown in figure 2. Thus we get a
new complex $\widehat{X}^{(1)}$ (strongly proper homotopy
equivalent to $X'$) in which $C'_{n(1)}$ turns into a subcomplex
$\widehat{C}_1 \sub \widehat{X}^{(1)}$ which contains $C_1$ and
has connected intersection with each $T_{i_m}$, $1 \leq m \leq s$.
Observe that $\widehat{X}^{(1)} - int(\widehat{C}_1)$ is strongly
proper homotopy equivalent to $X - int(C_{n(1)})$, by construction.\\
\indent Let $C^{(1)}_1 \sub C^{(1)}_2 \sub \cdots \sub
\widehat{X}^{(1)}$ be the obvious filtration obtained from $C_1
\sub C_2 \sub \cdots \sub X$ (i.e., each $C_n$ ``expands" to
$C^{(1)}_n$), and take $N \geq 1$ such that $\widehat{C}_1 \sub
C^{(1)}_N$. We apply to $C^{(1)}_N \sub \widehat{X}^{(1)}$ the
same argument we applied to $C_1 \sub X$ so as to get a new
complex $\widehat{X}^{(2)}$ (via a strong proper homotopy
equivalence with $\widehat{X}^{(1)}$ which leaves fixed a compact
subcomplex containing $C^{(1)}_N \supset \widehat{C}_1$) and a
subcomplex $\widehat{C}_2 \supset C^{(1)}_N \supset \widehat{C}_1$
with the required properties. Iterating this process we obtain the
desired $2$-complex $\widehat{X}$ (as the limit of the complexes
$\widehat{X}^{(i)}$, $i \geq 1$) and a nice filtration
$\widehat{C}_1 \sub \widehat{C}_2 \sub \cdots \sub \widehat{X}$.
\end{proof}
\begin{remark} \label{tree-condition}
Note that the strong homotopy equivalence $X \lga \widehat{X}$
obtained in Lemma \ref{alter} maps every tree in $X$ to another
tree in $\widehat{X}$, by construction.
\end{remark}
\indent Next we introduce some notation. Let $G$ be a finitely
generated one-relator group and $P=\langle X;R \rangle$ be any
(finite) presentation of $G$ with a single defining relation
$R=Q^s$ ($s$ maximal) which is assumed to be a cyclically reduced
word. We denote by $K_P$ the standard (compact) $2$-dimensional
CW-complex associated to this presentation. Note that $K_P^1$ is a
bouquet of circles consisting of a $1$-cell $e_i$ for each element
of the basis $x_i \in X$, all of them sharing the single vertex in
$K_P$. Finally, $K_P$ is obtained from $K_P^1$ by attaching a
$2$-cell $d$ via a PL map $S^1 \lga K_P^1$ which is the
composition of the map $z \in S^1 \mapsto z^s \in S^1$ and a PL
map $f_Q : S^1 \lga K_P^1$ which spells out the word $Q$.
\begin{remark} \label{Magnus} Note that every lift in the universal cover $\tilde{d} \sub
\tilde{K}_P$ of the $2$-cell $d \sub K_P$ is a disk (as $R$ is
cyclically reduced). Moreover, by the Magnus' Freiheitssatz (see
\cite{LS,MKS}) every subcomplex of the $1$-skeleton $K^1_P$ not
containing all the $1$-cells involved in the relator $R$ lifts in
the universal cover $\tilde{K}_P$ to a disjoint union of trees.
\end{remark}
\begin{proposition} \label{prop} Given any finite one-relator group presentation
$P$ as above, the universal cover $\tilde{K}_P$ of the standard
2-complex $K_P$ is strongly proper homotopy equivalent to another
$2$-complex $\widehat{K}_P$ which admits a nice filtration
$\widehat{C}_1 \sub \widehat{C}_2 \sub \cdots \sub \widehat{K}_P$.
\end{proposition}
\begin{proof} The proof is modelled after
(\cite{DV}, Thm. 2.1), and the method we use goes back to Magnus
\cite{M}. With the notation above, we set $n= length(Q)$. Note
that if $n=1$ then $R={x_{i_0}}^{\pm s}$ for some $x_{i_0} \in X$,
and $\tilde{K}_P$ is the universal cover of the bouquet
$K_{{\Z}_s} \vee (\vee_{i \neq i_0} e_i)$, where $K_{{\Z}_s}$ is
the standard $2$-complex associated to the obvious presentation of
${\Z}_s$. The $2$-complex $\widehat{K}_P= \tilde{K}_P$ clearly
satisfies the required properties, and its fundamental pro-group
corresponds to the trivial tower. As an example, the universal
cover $\tilde{K}_P$ is depicted in figure 3 for the presentation
$P =\langle a,b ; a^2 \rangle$ of ${\Z}*{\Z}_2$, indicating in
dark color the first two subcomplexes of a (nice) filtration
$\widehat{C}_1 \sub \widehat{C}_2 \sub \cdots \sub \tilde{K}_P$.\\
\begin{figure}
\centerline{\psfig{figure=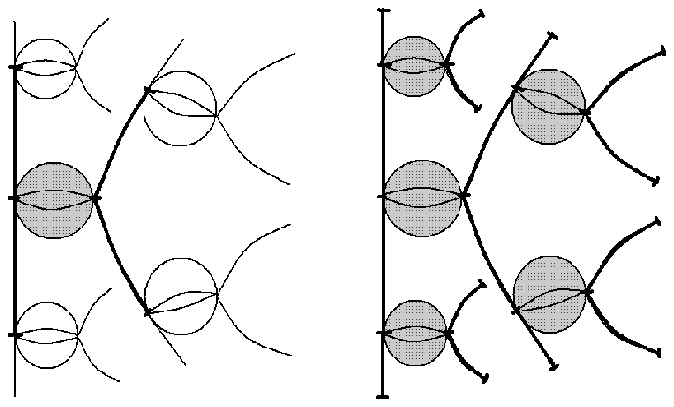,height=6cm,width=10cm}}
\vspace{-5mm} \label{figure3} \caption{}
\end{figure}
\indent Assume inductively the following two statements hold for
any (finite) group presentation $P' = \langle X' ; R' \rangle$
with a single defi\-ning relation $R'={Q'}^s$ ($s$ maximal) which
is a
cyclically reduced word with $k=length(Q') \leq n-1$ :\\
\begin{itemize}
\item[$(Case \; 1)_k$ :] If there is $x_{i_0} \in X'$ which occurs in
$Q'$ with exponent sum $\sigma_{Q'}(x_{i_0})=0$ then the
conclusion
of Proposition \ref{prop} follows for $\tilde{K}_{P'}$.\\
\item[$(Case \; 2)_k$ :] If the word $Q'$ contains no generators with
exponent sum zero then the conclusion of Proposition \ref{prop}
follows for $\tilde{K}_{P'}$.
\end{itemize}
Furthermore, we assume in addition that for any such $P'$ the
following condition is satisfied :
\begin{itemize}
\item[(*)] Every tree in the $1$-skeleton of $\tilde{K}_{P'}$ gets
mapped (under the strong proper homotopy equivalence) to another
tree in the final $2$-complex $\widehat{K}_{P'}$.
\end{itemize}

\indent Next we proceed to show case $k=n$ by proving that the
statements $(Case \; 1)_n$ and $(Case \; 2)_n$ are true as well as
condition (*) in each of them. Due to the rather long proofs of
both cases we deal with them separately in the two subsequent
sections.
\end{proof}
\begin{remark} It is well-known that if the one-relator $G$ is torsion-free
then the universal cover $\tilde{K}_P$ of $K_P$ is a contractible
$2$-dimensional $CW$-complex, by a result of Dyer and Vasquez
\cite{DV} which can be thought of as a geometric version of
Lyndon's Identity Theorem \cite{L} (see also \cite{Cock}).
Therefore, $\widehat{K}_P$ is also contractible and $\widehat{C}_1
\sub \widehat{C}_2 \cdots \sub \widehat{K}_P$ is a nice filtration
consisting of compact simply connected acyclic, and hence contractible, subcomplexes.\\
\indent In case $G$ has torsion we still find a contractible
subcomplex $\widehat{L}_P \sub \widehat{K}_P$ carrying the
fundamental pro-group (i.e., $pro-\pi_1(\widehat{L}_P) \cong
pro-\pi_1(\widehat{K}_P)$) and a nice filtration $\widehat{L}_1
\sub \widehat{L}_2 \sub \cdots \sub \widehat{L}_P$ consisting of
compact contractible subcomplexes. Indeed, fix a lift $\tilde{d}
\sub \tilde{K}_P$ of the $2$-cell $d \sub K_P$, and observe that
$\tilde{d}$ and all its translates $\tilde{d} \cdot g \sub
\tilde{K}_P$ (by the $G$-action) are copies of the disk $D^2$ (as
$R$ is cyclically reduced), which come together in subcomplexes
$L_k= \tilde{d} \cdot g_k \cup \tilde{d} \cdot (g_k Q) \cup \dots
\cup \tilde{d} \cdot (g_k Q^{s-1}) \sub \tilde{K}_P$, $g_k \in G$,
consisting of $s$ disks attached along their boundaries via the
identity map. It is worth noting that the only non-trivial
$2$-cycles in $\tilde{K}_P$ are those which are (finite)
combinations of non-trivial $2$-cycles in some of the $L_k$'s.
This is due to the fact that the relation module (associated with
the given presentation of $G$) is a cyclic ${\Z}G$-module with one
obvious relation (see \cite{Chis, Brown}). Next, we consider a
subcomplex $L_P \sub \tilde{K}_P$ containing only one $2$-cell
from each of the subcomplexes $L_k \sub \tilde{K}_P$ described
above. Note that then $L_P$ is contractible, and one can check
that the proof of Proposition \ref{prop} produces a new
$2$-complex $\widehat{L}_P \sub \widehat{K}_P$ (strongly proper
homotopy equivalent to $L_P$) together with a nice filtration
$\widehat{C}_1 \cap \widehat{L}_P \sub \widehat{C}_2 \cap
\widehat{L}_P \sub \cdots \sub \widehat{L}_P$ consisting of
compact contractible subcomplexes.
\end{remark}

\section{Proof that $(Case \; 1)_{\leq n-1} + (Case \; 2)_{\leq n-1} \Rightarrow (Case \; 1)_n$}
The purpose of this section is to prove $(Case \; 1)_n$. In fact,
we shall prove the conclusion is true for any (finite) one-relator
group presentation $P= \langle X;R \rangle$, $R=Q^s$ ($s$
maximal), assuming $R$ is cyclically reduced and there is $x_{i_0}
\in X$ which occurs in $Q$ with exponent sum $\sigma_Q(x_{i_0})=0$
such that the word obtained from $Q$ by deleting the symbols
$x_{i_0}$ and $x_{i_0}^{-1}$ has length $\leq n-1$. We shall
divide the proof into the following four steps
(we keep the notation for the cells of $K_P$ from $\S 2$).\\

\noindent {\bf 1. Notation and some preliminaries.} For each
integer $k$, let $J_k$ be a copy of the bouquet $\vee_{i \neq i_0}
e_i \sub K_P$ (with the $0$-cell of $K_P$ as base point) and take
$L$ to be the space obtained by attaching the space $J_k$ to the
real line ${\R}$ at each integer point $k$ (through its base
point). Let $h : L \lga L$ be the PL homeomorphism which takes a
point of $J_k$ to the corresponding point of $J_{k+1}$ and sends
$x \in {\R}$ to $x+1 \in {\R}$, and denote by $p : L \lga L/H
\equiv K_P^1$ the resulting covering map, where $H$ is the
infinite cyclic subgroup of self-homeomorphisms generated by $h$.
Moreover, one can check that there is a pointed PL map $f' : S^1
\lga L$ (taking $0 \in {\R}$ as the base point of $L$) such that
$pf'=f_Q$ ($=$the given PL map $S^1 \lga K_P^1$ which spells out
the word $Q$) and $f'$ spells out a cyclically reduced word $Q'$
(with $length(Q')<n$) in the free group $F(Y)$, where $Y$ is the
set of generators for $\pi_1(L, \{0\})$ obtained from those for
each space $J_k$ via the obvious base point changes (see
\cite{DV}). Let $u,v \in {\Z}$ be the integers such that $[u,v]
\sub {\R}$ is the smallest interval satisfying $\displaystyle
f'(S^1) \sub {\R} \cup \bigcup_{k \in [u,v]} J_k$. Let $K'_P$ be
the space obtained from $L$ by gluing disks $D^2_k$ via the
composition $f_k''$ of the map $z \mapsto z^s$ with the PL map
$f'_k=h^kf' : S^1 \lga L$, $k \in {\Z}$. We keep denoting by $H$
the infinite cyclic subgroup of self-homeomorphisms of $K'_P$
generated by the obvious extension of $h$. Thus, we have an
(intermediate) covering space $p : K'_P \lga K'_P/H \equiv K_P$
(corresponding to the kernel of the homomorphism $F(X)/N(R) \lga
{\Z}$ induced by $F(X) \lga {\Z}, w \mapsto \sigma_w(x_{i_0})$),
and we denote by $q : \tilde{K}_P \lga K'_P$ the corresponding
universal covering map. Let $K'_m \sub K'_P$ be the subcomplex
consisting of the real line ${\R}$ together with $\displaystyle
D_m^2 \cup_{f''_m} \left( \bigcup_{k \in [u+m,v+m]} J_k \right)$.
Note that, as the interval $[u+m, v+m]$ can be shrunk to a point,
$K'_m$ is homotopy equivalent to the wedge $K_{P'} \vee {\R}$
where $K_{P'}$ is the standard $2$-dimensional CW-complex
associated to a group presentation $P' = \langle X';{Q'}^s
\rangle$ where $Q'$ is as above and $X' \sub Y$ is the
subset consisting of those generators contained in $K'_m$.\\

\noindent {\bf 2. The structure of the universal cover
$\tilde{K}_P$.} It follows from (\cite{DV}, Sublemma 3.2.2) that
the inclusion $K'_m \sub K'_P$ ($m \in {\Z}$) induces an injection
of fundamental groups (this need not hold if $R=Q^s$ has not been
cyclically reduced, see Remark \ref{example}). In this way, each
component of the pre\-image $q^{-1}(K'_m) \sub \widetilde{K}_P$ is
a copy of the universal cover of $K'_m$. On the other hand, as
$K'_P = \displaystyle \bigcup_{m \in {\Z}} K'_m$, it is not hard
to see that the universal cover $\widetilde{K}_P$ then consists of
collections of copies of the universal covers of the spaces $K'_m$
($m \in {\Z}$), appropriately glued together along a certain
collection of trees. More precisely, observe that for e\-ve\-ry $m
> n$, $K'_m \cap K'_n = {\R}$ if $m-n > v-u$; otherwise, $K'_m
\cap K'_n \sub {\R} \cup J_{u+m} \cup \dots \cup J_{v+n}$ is a
graph with the property that the inclusion $K'_m \cap K'_n \sub
K'_m$ (or $K'_n$) induces an injection of the fundamental groups,
by the choice of the integers $u,v$. Thus, each component of the
preimage $q^{-1}(K'_m \cap K'_n) \sub \widetilde{K}_P$ is a copy
of the universal cover of $K'_m \cap K'_n$ which is a tree by an
application of the Magnus' Freiheitssatz to the presentation $P' =
\langle X';{Q'}^s \rangle$, as the generators of $X'$ which occur
in $K'_m \cap K'_n$ are not all of those involved in the relator
(see Remark \ref{Magnus}). Moreover, if more than two copies of
the universal covers of some of the subcomplexes $K'_m$ have
non-empty intersection in $\widetilde{K}_P$ then this intersection
must be a subtree of the intersection of any two of them.\\

\noindent {\bf 3. Altering $\widetilde{K}_P$ within its (strong)
proper homotopy type.} Next we check that the universal cover of
each complex $K'_m$ is strongly proper homotopy equivalent to
another $2$-complex $\widehat{K}'_m$ which admits a nice
filtration. To this end, observe that the universal cover
$\widetilde{K}_{P'}$ of $K_{P'}$ ($P'$ as in Step $1$) has this
property by the inductive hypothesis. Moreover, the universal
cover of $K'_m$ is obtained from the universal cover of $K_{P'}
\vee {\R}$ by ''expanding" in an appropriate way each vertex to an
interval $[u+m,v+m]$ so as to recover the subgraph $K'_m \cap L$.\\
\indent We perform in $\widetilde{K}_P$ all the elementary
internal collapses and/or expansions needed for passing from each
copy of the universal cover of $K'_m$ ($m \in {\Z}$) inside
$\widetilde{K}_P$ to the corresponding $2$-complex
$\widehat{K}'_m$. In this way we obtain a new $2$-complex
$\widehat{K}'_P$ (strongly proper homotopy equivalent to
$\widetilde{K}_P$ and still containing a copy of $q^{-1}({\R})$)
which consists of collections of copies of the complexes
$\widehat{K}'_m$ ($m \in {\Z}$) glued together appropriately along
the corresponding collection of trees obtained from those in the
construction of $\widetilde{K}_P$, according to condition (*) of
the inductive hypothesis (see $\S 2$). For every $m \in {\Z}$ and
every copy of $\widehat{K}'_m$ inside $\widehat{K}'_P$ we consider
a nice filtration $C_{1,m} \sub C_{2,m} \sub \cdots \sub
\widehat{K}'_m$ and denote by $T_1, T_2, \dots, T_s$ those trees
along which the given copy of $\widehat{K}'_m$ intersects with a
copy (inside $\widehat{K}'_P$) of any other subcomplex
$\widehat{K}'_n$ ($n \neq m$). Let $\widehat{C}_{1,m} \sub
\widehat{C}_{2,m} \sub \cdots \sub \widehat{K}_m$ be the
$2$-complex and the nice filtration obtained from $C_{1,m} \sub
C_{2,m} \sub \cdots \sub \widehat{K}'_m$ and the collection of
trees $T_1, T_2, \dots, T_s$ proceeding as in Lemma \ref{alter}.
In particular, each intersection $\widehat{C}_{n,m} \cap T_i$ ($n
\geq 1, 1 \leq i \leq s$) is either empty or connected (and
hence contractible).\\
\indent Let $\widehat{K}_P$ be the $2$-complex (strongly proper
homotopy equivalent to $\widehat{K}'_P$ and hence to
$\widetilde{K}_P$) obtained from collections of copies of the
complexes $\widehat{K}_m$ ($m \in {\Z}$) glued together as follows
: a copy of $\widehat{K}_m$ is being glued to a copy of
$\widehat{K}_n$ along a tree $T$ whenever the corresponding copies
of $\widehat{K}'_m$ and $\widehat{K}'_n$ inside $\widehat{K}'_P$
are glued together along the corresponding copy of the same tree
$T$.\\
\indent Note that in the process of altering $\widetilde{K}_p$ to
$\widehat{K}_P$ we have used inductive hypothesis and applications
of Lemma \ref{alter}. This way we ensure that
condition (*) gets satisfied (see Remark \ref{tree-condition}).\\

\noindent {\bf 4. Building the required filtration for
$\widehat{K}_P$.} We must now build a nice filtration for
$\widehat{K}_P$, and this will finish the proof. We keep the
notation from Step $3$. Fix $m \in {Z}$ and a copy of
$\widehat{K}_m$ in $\widehat{K}_P$, and consider
$\widehat{C}_{1,m} \sub \widehat{K}_m$. By abuse of notation, we
will denote by $\widehat{K}_{m(1)}, \dots , \widehat{K}_{m(p_1)}$
those different copies in $\widehat{K}_P$ of the corresponding
complexes which intersect the chosen copy of $\widehat{K}_m$ at
points of $\widehat{C}_{1,m}$ (i.e., if $m(i)=m(j)$ for some $1
\leq i < j \leq p_1$, the complexes $\widehat{K}_{m(i)}$,
$\widehat{K}_{m(j)}$ above are considered different copies of the
same complex). Take $N_1 \geq 1$ such that $\widehat{C}_{N_1,
m(j)} \cap \widehat{K}_m \supset \widehat{C}_{1,m} \cap
\widehat{K}_{m(j)}$ (both intersections being connected subtrees
by hypothesis), $1 \leq j \leq p_1$. Set $\widehat{C}_1 =
\widehat{C}_{1,m} \cup \displaystyle \left( \bigcup_{j=1}^{p_1}
\widehat{C}_{N_1, m(j)} \right) \sub \widehat{K}_P$ which is
easily shown to be simply connected, as the simply connected
compact subcomplexes $\widehat{C}_{\alpha, \beta}$ intersect each
other along connected subtrees of those trees used in the
construction of $\widehat{K}_P$ (see Step $3$). Moreover, by the
generalized Van Kampen's argument, as $\widehat{C}_{1,m}$ and each
$\widehat{C}_{N_1, m(j)}$ are members of a nice filtration of the
corresponding copy of $\widehat{K}_m$ and $\widehat{K}_{m(j)}$
respectively and (every copy of) any other subcomplex
$\widehat{K}_\lambda \sub \widehat{K}_P$ is simply connected, one
can check that (choosing base points on any given base ray in
$\widehat{K}_P$) the fundamental group of $\widehat{K}_P -
int(\widehat{C}_1)$ is indeed a free product of the fundamental
groups of $\widehat{K}_P - int(\widehat{C}_{1,m})$ and
$\widehat{K}_P - int(\widehat{C}_{N_1, m(j)})$ together with an
extra free group (of finite rank) coming from the intersection of
$\widehat{C}_{1,m}$ and each $\widehat{C}_{N_1, m(j)}$ with any
other subcomplex $\widehat{K}_\lambda \sub \widehat{K}_P$.\\
\indent Consider now $\widehat{C}_{2,m} \sub \widehat{K}_m$. We
denote by $\widehat{K}_{m(1)}, \dots , \widehat{K}_{m(p_2)}$ ($p_2
\geq p_1$) those different copies in $\widehat{K}_P$ of the
corresponding complexes which intersect the given copy of
$\widehat{K}_m$ at points of $\widehat{C}_{2,m}$, and take $N_2
\geq N_1$ such that $\widehat{C}_{N_2, m(j)} \cap \widehat{K}_m
\supset \widehat{C}_{2,m} \cap \widehat{K}_{m(j)}$ (both being
connected subtrees), $1 \leq j \leq p_2$. Next, for each $1 \leq j
\leq p_2$, we denote by $\widehat{K}_{m(j,1)}, \dots,
\widehat{K}_{m(j, q_j)}$ those different copies in $\widehat{K}_P$
of the corres\-pon\-ding complexes which intersect the given copy
of $\widehat{K}_{m(j)}$ at points of $\widehat{C}_{N_2, m(j)}$,
and take $N_3 \geq N_2$ such that $\widehat{C}_{N_3, m(j,l)} \cap
\widehat{K}_{m(j)} \supset \widehat{C}_{N_2, m(j)} \cap
\widehat{K}_{m(j,l)}$ (both being connected subtrees), $1 \leq l
\leq q_j$. Set $\widehat{C}_2 = \widehat{C}_{2,m} \cup
\displaystyle \left( \bigcup_{j=1}^{p_2} \widehat{C}_{N_2, m(j)}
\right) \cup \left( \bigcup_{j=1}^{p_2} \bigcup_{l=1}^{q_j}
\widehat{C}_{N_3, m(j,l)} \right) \sub \widehat{K}_P$ which is
again simply connected and contains $\widehat{C}_1$, by
construction.\\
\indent As before, one can also check that the fundamental group
of $\widehat{K}_P - int(\widehat{C}_2)$ is finitely generated and
free. Finally, repeating this process (starting off with the
successive $\widehat{C}_{n,m} \sub \widehat{K}_m$ and going each
time one step further inside $\widehat{K}_P$) we get a filtration
$\widehat{C}_1 \sub \widehat{C}_2 \sub \cdots \sub \widehat{K}_P$
of compact simply connected subcomplexes. Moreover, choosing base
points on any given base ray, one can easily check that
$rank(\pi_1(\widehat{K}_P - int(\widehat{C}_{n+1}))) \geq
rank(\pi_1(\widehat{K}_P - int(\widehat{C}_n)))$ and the
homomorphism $\pi_1(\widehat{K}_P - int(\widehat{C}_{n+1})) \lga
\pi_1(\widehat{K}_P - int(\widehat{C}_n))$ can be taken to be a
projection between finitely generated free groups.

\begin{remark} \label{lines}
Notice that the argument used for the (inductive) proof of $(Case
\; 1)_n$ shows that the inverse image $(q \circ p)^{-1}(e_{i_0})
\sub \widetilde{K}_P$ of the $1$-cell $e_{i_0} \sub K_P$ in Step
$1$ remains as a subcomplex of the final $2$-complex
$\widehat{K}_P$.
\end{remark}
\begin{remark} \label{example} Consider the finite group presentation
$P= \langle a,b ; a^{-1}ba^{-1}b^{-1}a \rangle$ with a single
defining relation $R=a^{-1}ba^{-1}b^{-1}a$ ($R=Q, s=1$) which is
not a cyclically reduced word in the free group $F(\{a,b\})$. Let
$K_P$ be the standard $2$-complex associated with this group
presentation, and let $K'_P$ and $K'_m$ ($m \in {\Z}$) be as above
(here, $x_{i_0}=b$). It is easy to see that in this case $K'_P =
\widetilde{K}_P$ and the inclusion $K'_m \sub K'_P$ does not
induce an injection between the fundamental groups (see figure 4).
On the other hand, this example shows a compact $2$-dimensional
CW-complex $K_P$ with $\pi_1(K_P) \cong {\Z}$ whose universal
cover can not be written as an increasing union of compact simply
connected subcomplexes.
\begin{figure}

\centerline{\psfig{figure=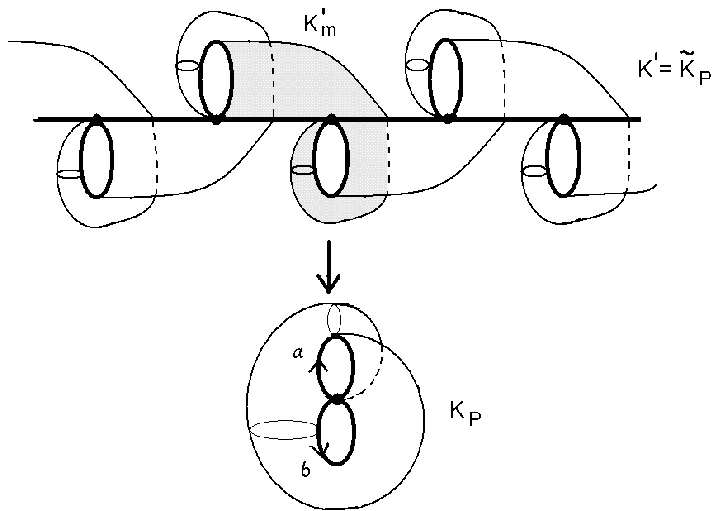,height=7cm,width=11cm}}
\label{figure4} \caption{}
\end{figure}
\end{remark}
\section{Proof that $(Case \; 1)_{\leq n-1} + (Case \; 2)_{\leq n-1} \Rightarrow (Case \; 2)_n$}
The purpose of this section is to prove $(Case \; 2)_n$ for any
(finite) one-relator group presentation $P= \langle X;R \rangle$
where $R=Q^s$ ($s$ maximal) is assumed to be a cyclically reduced
word with $length(Q)=n \geq 2$. Thus, suppose $Q$ contains no
generators
with exponent sum $0$. We divide the proof into the following three steps.\\

\noindent {\bf 1. Notation.} Since $length(Q) =n \geq 2$ it
follows that $Q$ must involve at least two generators $x_{i_0},
x_{i_1} \in X$. We introduce new symbols $A, B \notin F(X)$ and
take $X'= (X - \{x_{i_0} \}) \cup \{A\}$ and $X''= (X - \{x_{i_0},
x_{i_1} \}) \cup \{A, B\}$. Consider the group presentations $P' =
\langle X'; {Q'}^s \rangle$ and $P'' = \langle X''; {Q''}^s
\rangle$, where $Q'$ is obtained from $Q$ by replacing $x_{i_0}$
with $A^q$ ($q=\sigma_Q(x_{i_1}$), the exponent sum of $x_{i_1}$
in $Q$), and $Q''$ is obtained from $Q'$ by repla\-cing $x_{i_1}$
with $BA^{-p}$ ($p= \sigma_Q(x_{i_0})$). The words $Q', Q''$ are
not proper powers and $\sigma_{Q''}(A)=0$. Furthermore, we may
assume $Q''$ cyclically reduced, since the associated $2$-complex
$K_{P''}$ will only change up to homotopy and its universal cover
will only change up to strong proper
homotopy.\\

\noindent {\bf 2. The universal covers $\widetilde{K}_{P'}$,
$\widetilde{K}_{P''}$ and their relation with $\widetilde{K}_P$.}
Observe that the word obtained from $Q''$ by deleting the symbols
$A$ and $A^{-1}$ has length less than $n=length(Q)$, and hence the
induction hypothesis together with the argument used in $\S 3$
yields that the universal cover of $K_{P''}$ is strongly proper
homotopy equivalent to a $2$-complex $\widehat{K}_{P''}$ which
admits a nice filtration. According to \cite{DV}, $K_{P'}$ and
$K_{P''}$ are homotopy equivalent, and hence their universal
covers $\widetilde{K}_{P'}$ and $\widetilde{K}_{P''}$ are proper
homotopy equivalent. In fact, one can easily describe this proper
homotopy equivalence geo\-metrically as follows (see figure 5).
The complex $\widetilde{K}_{P'}$ is obtained from
$\widetilde{K}_{P''}$ by sliding the final endpoint $\beta(1)$ of
each lift $\beta \sub \widetilde{K}_{P''}$ of the (oriented)
generating circle corresponding to $B$ (dragging the material of
the $2$-cells involved, thus substituting the old one) over the
edge path $\alpha, \alpha \cdot A^{-1}, \dots, \alpha \cdot
A^{-p+1} \sub \widetilde{K}_{P''}$, where $\alpha$ is the lift of
the inverse path of the (oriented) generating circle corresponding
to $A$ whose initial endpoint $\alpha(0)$ coincides with
$\beta(1)$. Note that as $A$ generates an infinite cyclic subgroup
in both group presentations $P'$ and $P''$, the support of each of
the above edge paths is homeomorphic to a closed interval.
\begin{figure}
\centerline{\psfig{figure=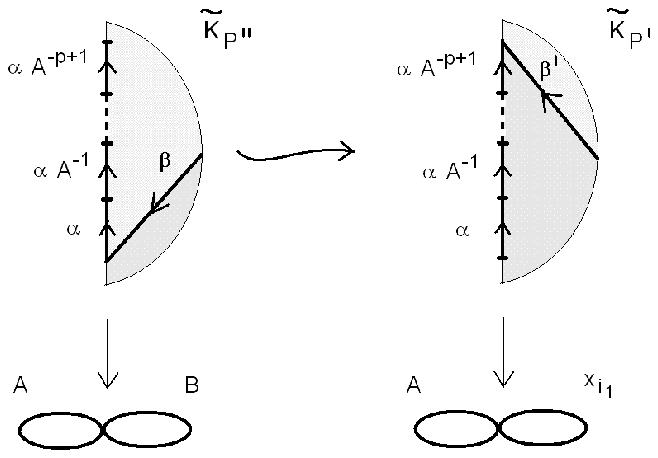,height=6cm,width=11cm}}
\label{figure5} \caption{}
\end{figure}
It is not hard to see from this description that
$\widetilde{K}_{P'}$ and $\widetilde{K}_{P''}$ are in fact
strongly proper homotopy equivalent. It
remains to show that $\widetilde{K}_P$ indeed has the required property.\\
\indent Let $e_{i_0} \sub K_P$ be the $1$-cell corresponding to
the generator $x_{i_0}$, and let $M$ be the mapping cylinder of a
map $e_{i_0} \lga S^1$ of degree $q$. Then, the adjunction complex
$W= K_P \cup_{e_{i_0}} M$ is homotopy equivalent to $K_{P'}$ and
hence their universal co\-vers $\widetilde{W}$ and
$\widetilde{K}_{P'}$ are proper homotopy equivalent. In fact,
$\widetilde{W}$ is built from copies of $\widetilde{K}_P$ and $Y_q
\times {\R}$ glued together appropriately along $Fr (Y_q \times
{\R})$, where $Y_q \sub {\R}^2$ consists of $q$ segments $[u,
v_i], 1 \leq i \leq q$, sharing a common vertex $u$. From here,
one gets the universal cover $\widetilde{K}_{P'}$ by shrinking
each copy of $Y_q \times {\R}$ to its centerline $\{u\} \times
{\R}$ (projecting down onto the base of $M$) which gets identified
with each connected component of $(q')^{-1}(A) \sub
\widetilde{K}_{P'}$, where $q' : \widetilde{K}_{P'} \lga K_{P'}$
is the universal covering map and $A \sub K_{P'}$ is the circle
(replacing $e_{i_0} \sub K_P$) representing the basis element $A
\in X'$ (see figure 6).
\begin{figure}
\centerline{\psfig{figure=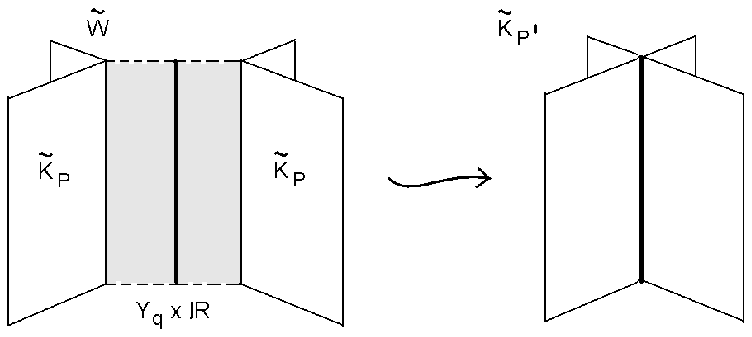,height=5cm,width=10cm}}
\vspace{-10mm} \label{figure6} \caption{}
\end{figure}
Notice that each component of $(q')^{-1}(A)$ is a line and the
subcomplex $(q')^{-1}(A) \sub \widetilde{K}_{P'}$ remains
unaltered when passing from $\widetilde{K}_{P'}$ to
$\widetilde{K}_{P''}$ and from $\widetilde{K}_{P''}$ to
$\widehat{K}_{P''}$ (see Remark \ref{lines}). It is easy to see
that
$\widetilde{W}$ is strongly proper homotopy equivalent to $\widetilde{K}_{P'}$.\\

\noindent {\bf 3. Altering $\widetilde{K}_P$ to $\widehat{K}_P$
and getting a nice filtration.} Choose a copy of $\widetilde{K}_P$
in $\widetilde{W}$ and denote by $\widehat{K}_P \sub
\widehat{K}_{P''}$ the $2$-complex that $\widetilde{K}_P$ ends up
being strongly proper homotopy equivalent to when performing all
the elementary internal collapses and/or expansions in order to
get $\widetilde{W} \sim \widetilde{K}_{P'} \sim
\widetilde{K}_{P''} \sim \widehat{K}_{P''}$ (here, ``$\sim$"
stands for strong proper homotopy equivalence). Observe that the
composition $\widetilde{W} \lga \widehat{K}_{P''}$ of these strong
homootopy equivalences maps every tree in $\widetilde{W}$ to
another tree in $\widehat{K}_{P''}$ (by construction), and hence
condition (*) gets satisfied (see $\S 2$). Finally, given a nice
filtration $\widehat{C}''_1 \sub \widehat{C}''_2 \sub \cdots \sub
\widehat{K}_{P''}$, we may assume that the intersection of each
$\widehat{C}''_i$ with every component of $(q')^{-1}(A)$ is either
empty or connected, by an application of Lemma \ref{alter} (and
replacing $\widehat{K}_{P''}$ if necessary). Given the above, one
can get a nice filtration $\widehat{C}_1 \sub \widehat{C}_2 \sub
\cdots \sub \widehat{K}_P$ where each $\widehat{C}_i$ can be
chosen to be the intersection $\widehat{C}''_i \cap
\widehat{K}_P$, $i \geq 1$. Indeed, one can check that the
fundamental group of each component $\widehat{J}$ of
$\widehat{K}_P - int(\widehat{C}_i)$ is a free factor of the
fundamental group of the corresponding component $\widehat{J}''
\supset \widehat{J}$ of $\widehat{K}_{P''} - int(\widehat{C}''_i)$
(as $cl(\widehat{J}'' - \widehat{J})$ intersects $\widehat{J}$
along a subcomplex of the collection of lines $(q')^{-1}(A)$), and
hence it is finitely generated and free by the Grushko-Neumann
theorem. Moreover, choosing base points on any given base ray, it
is not hard to check that $rank(\pi_1(\widehat{K}_P -
int(\widehat{C}_{i+1}))) \geq rank(\pi_1(\widehat{K}_P -
int(\widehat{C}_i)))$ and the homomorphism $\pi_1(\widehat{K}_P
-int(\widehat{C}_{i+1})) \lga \pi_1(\widehat{K}_P -
int(\widehat{C}_i))$ can be taken to be a projection between
finitely generated free groups.

\section{Appendix} \label{appendix}
This section is intended to provided the background and notation
needed in this paper, specially the notions of fundamental
pro-group and semistability at infinity. In what follows, we will
be working within the category $tow-Gr$ of {\it towers of groups}
whose objects are inverse sequences of groups
$$\underline{A} = \{ A_0 \stackrel{\phi_1}{\lgaf} A_1 \stackrel{\phi_2}{\lgaf} A_2 \lgaf \cdots \}$$
A morphism in this category will be called a {\it pro-morphism}.
See \cite{Geo, MarSe}
for a general reference.\\
\indent A tower $\underline{L}$ is a {\it free tower} if it is of
the form
$$\underline{L} = \{ L_0 \stackrel{i_1}{\lgaf} L_1 \stackrel{i_2}{\lgaf} L_2 \lgaf \cdots \}$$
where $L_i = \langle B_i \rangle$ are free groups of basis $B_i$
such that $B_{i+1} \sub B_i$, the differences $B_i - B_{i+1}$ are
finite and $\bigcap_{i=0}^{\infty} B_i = \emptyset$, and the
bonding homomorphisms $i_k$ are given by the corresponding basis
inclusions. On the other hand, a tower $\underline{P}$ is a {\it
telescopic tower} if it is of the form
$$\underline{P} = \{ P_0 \stackrel{p_1}{\lgaf} P_1 \stackrel{p_2}{\lgaf} P_2 \lgaf \cdots \}$$
where $P_i = \langle D_i \rangle$ are free groups of basis $D_i$
such that $D_{i-1} \sub D_i$, the differences $D_i - D_{i-1}$ are
finite (possibly empty),
and the bonding homomorphisms $p_k$ are the obvious projections.\\
\indent We will also use the full subcategory $(Gr, tow-Gr)$ of
$Mor(tow-Gr)$ whose objects are arrows $\underline{A} \lga G$,
where $\underline{A}$ is an object in $tow-Gr$ and $G$ is a group
regarded as a constant tower whose bonding maps are the identity.
Morphisms in $(Gr, tow-Gr)$ will also be
called pro-morphisms.\\
\indent From now on, $X$ will be a (strongly) locally finite
CW-complex. A proper map $\omega : [0, \infty) \lga X$ is called a
{\it proper ray} in $X$. We say that two proper rays $\omega,
\omega'$ {\it define the same end} if their restrictions
$\omega|_{\N}, \omega'|_{\N}$ are properly homotopic. Moreover, we
say that they {\it define the same strong end} if $\omega$ and
$\omega'$ are in fact properly homotopic. The CW-complex $X$ is
said to be semistable at infinity if any two proper rays
defining the same end also define the same strong end.\\
\indent Given a base ray $\omega$ in $X$ and a collection of
finite subcomplexes $C_1 \sub C_2 \sub \cdots \sub X$ so that $X =
\bigcup_{n=1}^{\infty} C_n$, the following tower, $pro-\pi_1(X,
\omega)$,
$$ \{ \pi_1(X, \omega(0)) \leftarrow \pi_1(X - int(C_1), \omega(t_1))
\leftarrow \pi_1(X - int(C_2), \omega(t_2)) \leftarrow \cdots \}$$
can be regarded as an object in $(Gr, tow-Gr)$ and it is called
the {\it fundamental pro-group of $(X, \omega)$}, where
$\omega([t_i, \infty)) \sub X - int(C_i)$ and the bonding
homomorphisms are induced by the inclusions. This tower does not
depend (up to pro-isomorphism) on the sequence of subcomplexes
$\{C_i\}_i$. It is worth mentioning that if $\omega$ and $\omega'$
define the same strong end, then  $pro-\pi_1(X, \omega)$ and
$pro-\pi_1(X, \omega')$ are pro-isomorphic. In particular, we may
always assume that $\omega$ is a cellular map. It is known that
$X$ is semistable at infinity if and only if $pro-\pi_1(X,
\omega)$ is pro-isomorphic to a tower where the bonding maps are
surjections. Moreover, in this case $\pi^{e}_1(X, \omega) =
{\limin} \; pro-\pi_1(X, \omega)$ is a well-defined useful
invariant which only depends (up to isomorphism) on the end
determined by $\omega$ (see \cite{GM2}). In a similar way, one can
define objects in $(Gr, tow-Gr)$ corresponding to the higher
homotopy pro-groups of
$(X, \omega)$.\\
\indent Finally, given a finitely presented group $G$ and a finite
$2$-dimensional CW-complex $X$ with $\pi_1(X) \cong G$, we say
that $G$ is semistable at infinity if the universal cover
$\tilde{X}$ of $X$ is so, and we will refer to the fundamental
pro-group of $\tilde{X}$ as the fundamental pro-group of $G$.

\section*{Acknowledgements}
The first three authors were supported by the project MTM
2007-65726. This research was also supported by Slovenian-Spanish
research grant BI-ES/04-05-014. The authors acknowledge remarks
and suggestions by Louis Funar.

\end{document}